\documentclass{amsart}

\usepackage{paralist}
\usepackage{color}
\usepackage{mathrsfs}
\usepackage{amssymb}
\usepackage{amsmath}
\usepackage{amsfonts}
\usepackage{stmaryrd}

\newtheorem*{thma}{Theorem~A}
\newtheorem*{thmb}{Theorem~B}

\newtheorem{theorem}{Theorem}[section]
\newtheorem{lemma}[theorem]{Lemma}
\newtheorem{proposition}[theorem]{Proposition}
\newtheorem{corollary}[theorem]{Corollary}
\newtheorem{claim}[theorem]{Claim}

\newtheorem{question}[theorem]{Question}

\theoremstyle{definition}
\newtheorem{definition}[theorem]{Definition}
\newtheorem{remark}[theorem]{Remark}
\newtheorem{fact}[theorem]{Fact}

\usepackage[pdftex]{hyperref}
\hypersetup{
    colorlinks=true, 
    linktoc=all,     
    linkcolor=blue,  
    allcolors=blue,
}

\usepackage[framemethod=tikz]{mdframed}

\newcommand{\dom}[1]{\ensuremath{\mathrm{dom}}(#1)}

\newcommand{\power}{\ensuremath{\mathscr{P}}}

\newcommand{\mc}{\mathcal}

\newcommand{\bb}{\mathbb}

\newcommand{\beq}{\begin{equation}}
\newcommand{\eeq}{\end{equation}}
\newcommand{\brm}{\begin{remark}\begin{rm}}
\newcommand{\erm}{\end{rm}\end{remark}}

\newcommand{\bce}{\begin{compactenum}}
\newcommand{\ece}{\end{compactenum}}

\newcommand{\cf}{\mathrm{cf}}

\newcommand{\pp}{\mathrm{pp}}
\newcommand{\ssup}{\mathrm{ssup}}

\newcommand{\R}{\bb{R}}
\newcommand{\Q}{\bb{Q}}
\renewcommand{\P}{\bb{P}}

\newcommand{\Coll}{\mathrm{Coll}}

\newcommand{\TP}{{\sf TP}}
\newcommand{\ITP}{{\sf ITP}}

\newcommand{\ISP}{{\sf ISP}}

\newcommand{\ZFC}{\sf ZFC}

\newcommand{\CH}{\sf CH}

\newcommand{\SCH}{{\sf SCH}}

\newcommand{\CP}{\sf CP}

\newcommand{\width}{\mathrm{width}}
\newcommand{\SNSP}{\sf SNSP}
\newcommand{\SSH}{\sf SSH}
\newcommand{\NSP}{{\sf NSP}}
\newcommand{\cNSP}{{\sf cNSP}}

\title{Narrow systems revisited}

\author{Chris Lambie-Hanson}
\address[Lambie-Hanson]{
Institute of Mathematics, 
Czech Academy of Sciences, 
{\v Z}itn{\'a} 25, Prague 1, 
115 67, Czech Republic
}
\email{lambiehanson@math.cas.cz}
\urladdr{https://users.math.cas.cz/~lambiehanson/}

\thanks{The author was supported by the Czech Academy of Sciences 
(RVO 67985840) and the GA\v{C}R project 23-04683S}

\begin{document}

\begin{abstract}
  Motivated by two open questions about two-cardinal tree properties, we introduce and study 
  generalized narrow system properties. The first of these questions asks whether the strong tree 
  property at a regular cardinal $\kappa \geq \omega_2$ implies the Singular Cardinals Hypothesis 
  ($\SCH$) above $\kappa$. We show here that a certain narrow system property at $\kappa$ 
  that is closely related to the strong tree property, and holds in all known models thereof, suffices to 
  imply $\SCH$ above $\kappa$. The second of these questions asks whether the strong tree property 
  can consistenty hold simultaneously at all regular cardinals $\kappa \geq \omega_2$. We show 
  here that the analogous question about the generalized narrow system property has a positive 
  answer. We also highlight some connections between generalized narrow system properties and 
  the existence of certain strongly unbounded subadditive colorings.
\end{abstract}


\maketitle

\section{Introduction}

A significant line of research in modern combinatorial set theory concerns the study of compactness 
principles that hold at (and sometimes characterize) large cardinals, the extent to which these 
compactness principles can hold at smaller cardinals, and the extent to which these principles can 
be said to capture the ``essence" of the respective large cardinal. To take a classical 
example, the tree property characterizes weakly compact cardinals among strongly inaccessible 
cardinals, while Mitchell \cite{mitchell} showed that the tree property at $\aleph_2$ is equiconsistent 
with the existence of a weakly compact cardinal. A number of questions remain open, though, about 
the extent to which the tree property can hold at smaller cardinals. The most prominent, due to Magidor, 
asks whether it is consistent that the tree property holds simultaneously at all regular cardinals 
greater than or equal to $\aleph_2$.

Generalizations of the tree property, known collectively as \emph{two-cardinal tree properties}, were 
introduced in the 1970s by Jech \cite{jech_combinatorial_problems} and Magidor 
\cite{magidor_combinatorial_characterization} to provide combinatorial characterizations of strongly 
compact and supercompact cardinals. Let us now recall some of the important definitions, in their 
modern formulation.

\begin{definition} \label{tp_def}
  Suppose that $\kappa \leq \lambda$ are uncountable cardinals, with $\kappa$ regular. A 
  $(\kappa, \lambda)$-tree is a structure $\mc T = \langle T_x \mid x \in \power_\kappa \lambda \rangle$ 
  such that
  \begin{itemize}
    \item for all $x \in \power_\kappa \lambda$, $T_x$ is a nonempty collection of subsets of $x$;
    \item for all $x \subseteq y \in \power_\kappa \lambda$ and all $t \in T_y$, we have $t \cap x \in T_x$.
  \end{itemize}
  A $(\kappa, \lambda)$-tree $\mc T$ is \emph{thin} if $|T_x| < \kappa$ for all $x \in \power_\kappa \lambda$.
  A \emph{cofinal branch} through $\mc T$ is a set $b \subseteq \lambda$ such that $b \cap x \in T_x$ for all 
  $x \in \power_\kappa \lambda$.
  
  The \emph{$(\kappa, \lambda)$-tree property}, denoted $\TP(\kappa, \lambda)$, is the assertion that every 
  thin $(\kappa, \lambda)$-tree has a cofinal branch.
  The \emph{ineffable $(\kappa, \lambda)$-tree property}, denoted $\ITP(\kappa, \lambda)$, is the assertion 
  that, for every thin $(\kappa, \lambda)$-tree $\mc T = \langle T_x \mid x \in \power_\kappa \lambda \rangle$ 
  and every choice function $d \in \prod_{x \in \power_\kappa \lambda} T_x$, there is a set $b \subseteq 
  \lambda$ such that the set
  \[
    \{x \in \power_\kappa \lambda \mid b \cap x = d(x)\}
  \] 
  is stationary in $\power_\kappa \lambda$. 
  
  The \emph{strong tree property} at $\kappa$, denoted $\TP_\kappa$, is the assertion that 
  $\TP(\kappa, \lambda)$ holds for all $\lambda \geq \kappa$.\footnote{To head off potential confusion, we note that 
  $\TP_\kappa$ is stronger than the classical tree property at $\kappa$, which is typically denoted $\TP(\kappa)$ 
  and is equivalent to $\TP(\kappa, \kappa)$ in our notation.} The \emph{super tree property} at $\kappa$, 
  denoted $\ITP_\kappa$, is the assertion that $\ITP(\kappa, \lambda)$ holds for all $\lambda 
  \geq \kappa$.
\end{definition}

\begin{fact}
  Suppose that $\kappa$ is an inaccessible cardinal.
  \begin{itemize}
    \item (Jech \cite{jech_combinatorial_problems}) $\kappa$ is strongly compact if and only if 
    $\TP_\kappa$ holds.
    \item (Magidor \cite{magidor_combinatorial_characterization}) $\kappa$ is supercompact if and only if 
    $\ITP_\kappa$ holds.
  \end{itemize}
\end{fact}

The modern study of two-cardinal tree properties at accessible cardinals did not begin until the 2000s, when 
the relevant definitions (including, e.g., the notion of a \emph{thin} $(\kappa, \lambda)$-tree 
introduced above) were isolated by Wei\ss\ \cite{weiss}. Since then, they have been the focus of a large 
amount of research, a considerable amount of which has been directed toward the study of their influence 
on cardinal arithmetic. Most notably, results of Viale \cite{viale_guessing_models} and 
Krueger \cite{krueger_sch} together show that, for a regular cardinal $\kappa \geq \omega_2$, 
$\ISP_\kappa$, which is a strengthening of $\ITP_\kappa$ also introduced by Wei\ss\ in \cite{weiss}, 
implies the Singular Cardinals Hypothesis ($\SCH$) above $\kappa$. In 
\cite[Theorem A]{arithmetic_paper}, the author and Stejskalov\'{a} show that $\SCH$ above $\kappa$ (and in fact 
Shelah's Strong Hypothesis ($\SSH$), a strengthening of $\SCH$, above $\kappa$) already follows 
from a significant weakening of $\ISP_\kappa$ that holds if, e.g., $\kappa$ is strongly compact or 
if $\kappa = \omega_2$ and we are in an extension by Mitchell forcing starting with a strongly 
compact cardinal.

We note also the seminal result of Solovay \cite{solovay_sch} stating that if $\kappa$ is a strongly 
compact cardinal, then $\SCH$ holds above $\kappa$. Recalling that, among inaccessible cardinals, 
$\TP_\kappa$ characterizes strongly compact cardinals whereas $\ITP_\kappa$ 
characterizes supercompact cardinals, this, together with the 
results mentioned in the previous paragraph, naturally leads to the following question, already 
asked in, e.g., \cite{hachtman_sinapova_itp} and \cite{ineffable_tree_property}:

\begin{question} \label{sch_q}
  Suppose that $\kappa \geq \omega_2$ is a regular cardinal. Does $\ITP_\kappa$ (or $\TP_\kappa$) imply 
  $\SCH$ above $\kappa$?
\end{question}

The analogue of Magidor's question is also of interest for these two-cardinal tree properties (see 
\cite{ineffable_tree_property} for more discussion of this question):

\begin{question} \label{global_q}
  Is it consistent that $\ITP_\kappa$ (or $\TP_\kappa$) holds for all regular cardinals $\kappa \geq \omega_2$?
\end{question}

Questions \ref{sch_q} and \ref{global_q} have a tight connection: by a theorem of Specker \cite{specker}, 
if $\mu$ is a cardinal and $2^\mu = \mu^+$, then the tree property fails at $\mu^{++}$ and therefore, 
\emph{a fortiori}, $\TP_{\mu^{++}}$ fails. Thus, if $\mu$ is a singular strong limit cardinal and 
$\SCH$ holds at $\mu$, then $\TP_{\mu^{++}}$ fails, so a positive answer to Question \ref{sch_q} 
would entail a negative answer to Question \ref{global_q}.

Motivated by these question, we prove here some results that we feel hint at a positive answer to Question 
\ref{sch_q} or at least 
indicate that genuinely new ideas would be needed to establish a negative answer. This work will involve 
introducing and analyzing generalizations of the classical notion of a \emph{narrow $\kappa$-system}, 
introduced by Magidor and Shelah \cite{magidor_shelah} to facilitate study of the tree property, particularly 
at successors of singular cardinals. Narrow systems have continued to play a central role in the 
study of the tree property, with verifications of the tree property at a cardinal $\kappa$ (especially 
when $\kappa$ is the successor of a singular cardinal) typically at least implicitly consisting of 
the following two-step process:
\begin{enumerate}
  \item Prove that every $\kappa$-tree has a narrow $\kappa$-subsystem.
  \item Prove that every narrow $\kappa$-system has a cofinal branch. The cofinal branch through the 
  subsystem identified in step (1) then generates a branch through the given $\kappa$-tree.
\end{enumerate}

In this paper, we generalize the notion of narrow system from the setting of cardinals $\kappa$ to 
arbitrary directed partial orders $\Lambda$ and show that these generalized system can play the same 
role in the study of generalized tree properties that narrow $\kappa$-systems play in the study 
of the classical tree property at $\kappa$. We introduce the generalized narrow system properties 
$\mathsf{NSP}(\Lambda)$, asserting that every narrow $\Lambda$-system has a cofinal branch, and study these 
properties, particularly in relation to their connections to Questions \ref{sch_q} and \ref{global_q}.

In Section \ref{concrete_sec}, before introducing narrow systems in their full generality, we define 
a specific type of system, which we call a \emph{concrete $\power_\kappa \lambda$-system}, 
that is particularly relevant to the study of $\TP(\kappa, \lambda)$. The narrow system property 
$\cNSP(\power_\kappa \lambda)$ then asserts that every narrow concrete $\power_\kappa \lambda$-system 
has a cofinal branch, and $\cNSP_\kappa$ denotes the assertion that $\cNSP(\power_\kappa \lambda)$ 
holds for all $\lambda \geq \kappa$.
$\cNSP_\kappa$ holds in all known models of $\TP_\kappa$, 
and is often used, at least implicitly, in verifications that $\TP_\kappa$ holds in a given model, 
especially if $\kappa$ is the successor of a singular cardinal. At the same time, we show that this narrow 
system property is strong enough to imply instances of $\SSH$:

\begin{thma}
  Suppose that $\kappa \geq \omega_2$ is a regular cardinal and $\cNSP_\kappa$ holds. 
  Then $\SSH$ holds above $\kappa$.
\end{thma}

Then, in Section \ref{system_sec}, we introduce the notion of narrow $\Lambda$-system and the 
narrow $\Lambda$-system property ($\mathsf{NSP}(\Lambda)$) for an arbitrary directed partial order 
$\Lambda$ and, as an illustration of their utility, 
use them to prove that generalized tree properties hold at successors of 
singular limits of strongly compact cardinals (Theorem \ref{strongly_compact_tp_thm}). 

In Section \ref{subbadditive_sec}, we connect narrow $\Lambda$-systems with 
strongly unbounded subadditive colorings, proving both that instances of $\NSP(\Lambda)$ entail 
the nonexistence of such functions on $\Lambda^{[2]}$ and, in turn, that the nonexistence of such 
functions on $(\power_\kappa \lambda)^{[2]}$ can be used in place of $\cNSP_\kappa$ 
in the hypothesis of Theorem A (cf.\ Corollaries \ref{nsp_sub_cor} and \ref{sub_ssh_cor}, 
respectively).

The remainder of the paper is devoted to a global consistency result showing that Question \ref{global_q} 
has a positive answer if the two-cardinal tree properties are replaced by generalized narrow system 
properties:
\begin{thmb}
  Suppose that there is a proper class of supercompact cardinals. Then there is a (class) forcing extension 
  in which $\NSP(\Lambda)$ holds for every directed partial order $\Lambda$.
\end{thmb}
Section \ref{preservation_sec} contains the proof of a technical branch preservation lemma for 
generalized narrow systems, and then Section \ref{consistency_sec} applies this lemma to prove  
Theorem B.

\subsection{Notational conventions} Unless otherwise noted, we follow standard set theoretic notational conventions. $\mathrm{On}$ denotes the class of all ordinals.
Given an infinite cardinal $\kappa$ and a set $X$, $\power(X)$ denotes the power set of $X$, and $\power_\kappa X$ 
denotes $\{x \subseteq X \mid |x| < \kappa\}$. If $x$ is a set of ordinals, then the \emph{strong supremum} of $x$ is 
the ordinal $\ssup(x) := \sup\{\alpha + 1 \mid \alpha \in x\}$, i.e., $\ssup(x)$ is the least ordinal $\beta$ such that 
$\alpha < \beta$ for all $\alpha \in x$. Given a partial order $\Lambda$, we let $\Lambda^{[2]}$ denote the 
set of ordered pairs $(u,v)$ from $\Lambda$ such that $u <_\Lambda v$. Sets of the form 
$\power_\kappa X$ will be interpreted as partial orders with the order relation given by $\subsetneq$. 
In particular, $(\power_\kappa X)^{[2]}$ denotes the set of pairs $(x,y)$ of elements of 
$\power_\kappa X$ with $x \subsetneq y$.

\section{Concrete systems} \label{concrete_sec}

Before we introduce the general notion of \emph{(narrow) $\Lambda$-system} for an arbitrary directed 
order $\Lambda$, and in order to help motivate the more abstract general definition, 
we first consider an important special case. 

\begin{definition} \label{concrete_system_def}
  Suppose that $\kappa \leq \lambda$ are uncountable cardinals, with $\kappa$ regular. A \emph{concrete $\power_\kappa 
  \lambda$-system} is a structure $\mc S = \langle S_x \mid x \in A \rangle$ such that
  \begin{enumerate}
    \item $A$ is a $\subseteq$-cofinal subset of $\power_\kappa \lambda$;
    \item for all $x \in A$, $\emptyset \neq S_x \subseteq \power(x)$;
    \item for all $x \subseteq y$, both in $A$, there is $t \in S_y$ such that 
    $t \cap x \in S_x$.
  \end{enumerate}
  The \emph{width} of $\mc S$ is defined to be $\width(\mc S) := \sup\{|S_x| \mid x \in A\}$. We say that 
  $\mc S$ is a \emph{narrow} concrete $\power_\kappa \lambda$-system if $\width(\mc S)^+ < \kappa$.
  A \emph{cofinal branch} through $\mc S$ is a set $b \subseteq \lambda$ such that the set
  $\{x \in A \mid b \cap x \in S_x\}$ is $\subseteq$-cofinal in $\power_\kappa \lambda$.
\end{definition}

Classical narrow systems, with levels indexed by ordinals, were introduced by Magidor and Shelah 
in \cite{magidor_shelah} as a central tool in the study of the tree property, particularly at successors of 
singular cardinals. Indeed, all known verifications of the tree property at the successor of a singular 
cardinal $\mu$ at least implicitly go through the following two steps:
\begin{enumerate}
  \item Show that that every $\mu^+$-tree $\mc T$ has a narrow subsystem $\mc S$ of height $\mu^+$.
  \item Show that every narrow system of height $\mu^+$ has a cofinal branch; in particular, 
  $\mc S$ has a cofinal branch, which gives rise to a cofinal branch through $\mc T$.
\end{enumerate}

One of the motivating observations for this paper is that narrow concrete $\power_\kappa \lambda$-systems 
play an analogous role for $(\kappa, \lambda)$-trees. For example, by an analogue of the 
two-step argument outlined above, we can show that the two-cardinal tree property 
$\TP_\kappa$ holds if $\kappa$ is the successor of a singular limit of strongly 
compact cardinals. Since a more general version of this statement is true, we 
postpone its proof until after we introduce the more general definition of ``narrow 
system"; it follows as a special case of Theorem \ref{strongly_compact_tp_thm} 
below.

We now turn to showing that the existence of cofinal branches through certain narrow 
concrete systems implies instances of $\SCH$ (and $\SSH$). To state the 
results concisely, we introduce the following terminology.

\begin{definition} \label{concrete_prop_def}
  Let $\kappa$ be a regular uncountable cardinal. For a cardinal $\lambda \geq \kappa$, we say 
  that the \emph{concrete narrow $\power_\kappa \lambda$-system property} holds 
  (denoted $\cNSP(\power_\kappa \lambda)$) if every narrow concrete
  $\power_\kappa \lambda$-system has a cofinal branch. We say that $\cNSP_\kappa$ holds 
  if $\cNSP(\power_\kappa \lambda)$ holds for all $\lambda \geq \kappa$.
\end{definition}

\begin{remark} \label{remark_23}
  It is worth taking the time to compare Definitions \ref{concrete_system_def} and Definition 
  \ref{concrete_prop_def} with Definition \ref{tp_def}, as the definitions of narrow concrete 
  $\power_\kappa \lambda$-systems and thin $(\kappa, \lambda)$-trees are quite similar. 
  The two salient differences are:
  \begin{itemize}
    \item The definition of narrow concrete $\power_\kappa \lambda$-system is more restrictive 
    with regards to the size of each level, requiring $\width(\mc S)^+ < \kappa$, whereas 
    a thin $(\kappa, \lambda)$-tree $\mc T$ is only required to satisfy $|T_x| < \kappa$ for all 
    $x \in \power_\kappa \lambda$.
    \item On the other hand, the definition of thin $(\kappa, \lambda)$-tree is more restrictive 
    with regards to the \emph{coherence properties} of the structure, requiring that, for all 
    $x \subseteq y$ and \emph{all} $t \in T_y$, we have $t \cap x \in T_x$, whereas the analogous 
    requirement in the definition of narrow concrete $\power_\kappa \lambda$-system only 
    requires the existence of one such $t$.
  \end{itemize}
  Therefore, it is not immediately evident whether either $\TP_\kappa$ or $\cNSP_\kappa$ 
  implies the other, though we shall see that, in practice, $\cNSP_\kappa$ 
  is easier to arrange than, and does not imply, $\TP_\kappa$ (cf.\ Remark 
  \ref{cnsp_tp_remark} below). The question of whether
  $\TP_\kappa$ implies $\cNSP_\kappa$ remains open and very much 
  of interest.
\end{remark}

Our verifications of $\SSH$ will go through the machinery of \emph{covering 
matrices} introduced by Viale in his proof that $\SCH$ follows from the 
Proper Forcing Axiom \cite{viale_pfa_sch}.

\begin{definition}
	Let $\theta < \lambda$ be regular cardinals. A \emph{$\theta$-covering matrix
		for $\lambda$} is a matrix $\mathcal{D} = \langle D(i, \beta) \mid i <
	\theta, ~ \beta < \lambda \rangle$ such that:
	\begin{enumerate}
		\item for all $\beta < \lambda$, $\langle D(i, \beta) \mid i < \theta \rangle$
		is a $\subseteq$-increasing sequence and $\bigcup_{i < \theta} D(i, \beta) =
		\beta$;
		\item for all $\beta < \gamma < \lambda$ and $i < \theta$, there is $j <
		\theta$ such that $D(i, \beta) \subseteq D(j, \gamma)$.
	\end{enumerate}
\end{definition}

We will be especially interested in covering matrices satisfying certain additional properties.

\begin{definition}
	Suppose that $\theta < \lambda$ are regular cardinals and $\mc D$ is a
	$\theta$-covering matrix for $\lambda$.
	\begin{enumerate}
		\item $\mc D$ is \emph{transitive} if, for all $\beta < \gamma < \lambda$
		and all $i < \theta$, if $\beta \in D(i, \gamma)$, then
		$D(i, \beta) \subseteq D(i, \gamma)$.
		\item $\mc D$ is \emph{uniform} if, for every limit ordinal $\beta < \lambda$,
		there is $i < \theta$ such that $D(i, \beta)$ contains a club in $\beta$.
		\item $\CP(\mc D)$ holds if there is an unbounded $A \subseteq \lambda$
		such that $[A]^\theta$ is covered by $\mc D$, i.e., for all
		$X \in [A]^\theta$, there are $\beta < \lambda$ and $i < \theta$ for which
		$X \subseteq D(i, \beta)$.
	\end{enumerate}
\end{definition}

The following theorem is proven in \cite{arithmetic_paper} (it was previously known 
in the case in which $\mu$ is strong limit (cf.\ \cite[Lemma 6]{viale_covering})).

\begin{theorem}[\cite{arithmetic_paper}] \label{downward_coherence_thm}
  Suppose that $\mu$ is a singular cardinal, $\theta = \cf(\mu)$, and 
  $\mc D$ is a uniform, transitive $\theta$-covering matrix for $\mu^+$. 
  Then, for every $x \in \power_\mu \mu^+$, there is $\gamma_x < \mu^+$ such that, 
  for all $\beta \in [\gamma_x, \mu^+)$, there is $i < \theta$ such that, for all 
  $j \in [i, \theta)$, we have $x \cap D(j,\beta) = x \cap D(j, \gamma_x)$.
\end{theorem}

We will also need to recall some basic background about Shelah's Strong Hypothesis. 
$\SSH$ is the assertion that $\pp(\mu) = \mu^+$ for every singular cardinal $\mu$, where 
$\pp(\mu)$ denotes the \emph{pseudopower} of $\mu$. For a cardinal $\kappa$, we say that 
$\SSH$ holds \emph{above $\kappa$} if $\pp(\mu) = \mu^+$ for every singular cardinal $\mu > \kappa$. 
For our purposes, we will not need to recall the definition of $\pp(\mu)$; the following facts will 
suffice:

\begin{fact} \label{pp_fact}
In what follows, if $\vec{\mu} = \langle \mu_i \mid i < \theta \rangle$ is a sequence of regular cardinals, 
then $\prod \vec{\mu}$ denotes the set of functions $f$ such that $\dom{f} = \theta$ and $f(i) < \mu_i$ 
for all $i < \theta$. Given $f,g \in \prod{\vec{\mu}}$, we say that $f <^* g$ if there is $i < \theta$ such 
that $f(j) < g(j)$ for all $j \in [i, \theta)$. The second and third facts below are both implicit in \cite{cardinal_arithmetic}; the cited references provide more explicit explanations.
\begin{enumerate}
  \item \cite[\S 2, Claim 2.4]{cardinal_arithmetic} If $\mu$ is a singular cardinal of uncountable cofinality and 
  $\{\nu < \mu \mid \pp(\nu) = \nu^+\}$ is stationary in $\mu$, then $\pp(\mu) = \mu^+$.
  \item \cite[Observation 4.4]{matet_meeting_numbers}
  Suppose that $\mu$ is a singular cardinal and $\pp(\mu) > \mu^+$. Then there is an increasing 
  sequence of regular cardinals $\vec{\mu} = \langle \mu_i \mid i < \cf(\mu) \rangle$ converging to 
  $\mu$ such that $\cf(\prod \vec{\mu}, <^*) > \mu^+$.
  \item \cite[Proposition 4.18]{arithmetic_paper} Let $\kappa$ be an infinite cardinal such that 
  $\SSH$ holds above $\kappa$. Then $\SCH$ holds above $\kappa$.
\end{enumerate}
\end{fact}

The connection between covering matrices and $\SSH$ comes via the following result.

\begin{theorem} \label{cp_ssh_thm} \cite[Theorem 4.19]{arithmetic_paper}
  Suppose that $\mu$ is a singular cardinal, $\theta = \cf(\mu)$, and $\vec{\mu} = \langle \mu_i \mid 
  i < \theta \rangle$ is an increasing sequence of regular cardinals converging to $\mu$. Suppose moreover 
  that $\mc D = \langle D(i,\beta) \mid i < \theta, ~ \beta < \mu^+ \rangle$ is a $\theta$-covering matrix 
  for $\mu^+$ such that
  \begin{enumerate}
    \item for all $i < \theta$ and $\beta < \mu^+$, we have $|D(i,\beta)| < \mu_i$;
    \item $\CP(\mc D)$ holds.
  \end{enumerate}
  Then $\cf(\prod \vec{\mu}, <^*) = \mu^+$.
\end{theorem}

We are now ready for the main result of this section, which will then yield Theorem A.

\begin{theorem} \label{nsp_cp_thm}
  Suppose that $\mu$ is a singular cardinal, $\theta = \cf(\mu)$, and there is a 
  regular cardinal $\kappa \in [\theta^{++}, \mu)$ such that
  $\cNSP(\power_\kappa \mu^+)$ holds. Then $\CP(\mc D)$ holds for every 
  uniform, transitive $\theta$-covering matrix $\mc D$ for $\mu^+$.
\end{theorem}

\begin{proof}
  Let $\lambda := \mu^+$.
  Fix a uniform, transitive $\theta$-covering matrix $\mc D = \langle D(i,\beta) 
  \mid i < \theta, ~ \beta < \lambda \rangle$ for $\lambda$, and let 
  $A := \{x \in \power_\kappa \lambda \mid \cf(\ssup(x)) > \theta \}$. Since 
  $\kappa > \theta^+$ is a regular cardinal, $A$ is cofinal in $\power_\kappa \lambda$; note that, 
  for $x \in A$, we have $\sup(x) = \ssup(x)$. 
  For each $x \in A$, let $\gamma_x < \mu^+$ be the least ordinal satisfying the 
  conclusion of Theorem~\ref{downward_coherence_thm}; namely, for all $\beta \in 
  [\gamma_x, \lambda)$ and all sufficiently large $j < \theta$, we have $x \cap 
  D(j, \beta) = x \cap D(j, \gamma_x)$. Note that we must have 
  $\gamma_x \geq \sup(x)$ and, if $x \subseteq y$ are both in $A$, then 
  $\gamma_x \leq \gamma_y$.
  
  For each $x \in A$, let 
  \[
    S_x := \{x \cap D(i, \gamma_x) \mid i < \theta \text{ and } \sup(x \cap D(i, 
    \gamma_x)) = \sup(x)\}.
  \]    
  Since $\cf(\sup(x)) > \theta$ and $x = \bigcup_{i < \theta} (x \cap D(i, \gamma_x))$, 
  it must be the case that $x \cap D(i, \gamma_x) \in S_x$ for all sufficiently large 
  $i < \theta$.
  
  We claim that $\mc S = \langle S_x \mid x \in A \rangle$ is a concrete 
  $\power_\kappa \lambda$-system. The only nontrivial condition to check is 
  clause (3) of Definition~\ref{concrete_system_def}. To this end, fix 
  $x \subseteq y$, both in $A$. We know that, for all sufficiently large $i < \theta$, 
  we have
  \begin{itemize}
    \item $\sup(x \cap D(i, \gamma_x)) = \sup(x)$;
    \item $\sup(y \cap D(i, \gamma_y)) = \sup(y)$;
    \item either $\gamma_x = \gamma_y$ or $\gamma_x \in D(i, \gamma_y)$; in either case, 
    since $\mc D$ is transitive, we have $D(i, \gamma_x) \subseteq D(i, \gamma_y)$.
  \end{itemize}
  Therefore, choosing $i < \theta$ sufficiently large, we have 
  $y \cap D(i, \gamma_y) \in S_y$ and 
  \[
    (y \cap D(i, \gamma_y)) \cap x = D(i, \gamma_y) \cap x = D(i, \gamma_x) \cap x 
    \in S_x,
  \]
  where the second equality holds by the choice of $\gamma_x$. Therefore,
  we have found $t \in S_y$ for which $t \cap x \in S_x$, as desired.
  
  Moreover, we have $|S_x| \leq \theta$ for all $x \in A$, so $\mc S$ is a 
  \emph{narrow} concrete $\power_\kappa \lambda$-system. We can therefore apply
  $\cNSP(\power_\kappa \lambda)$ to find a cofinal branch $b$ through $\mc S$.
  
  \begin{claim}
    $b$ is unbounded in $\lambda$.
  \end{claim}
  
  \begin{proof}
    Fix $\alpha < \lambda$; we will show that $b \setminus \alpha$ is nonempty. 
    Find $x \in A$ such that $\alpha \in x$ and $b \cap x \in S_x$. By the definition 
    of $S_x$, it follows that $\sup(b \cap x) = \sup(x) > \alpha$.
  \end{proof}
  
  We will therefore be done if we show that $[b]^\theta$ is covered by 
  $\mc D$, as then $b$ will witness $\CP(\mc D)$. To this end, fix $z \in [b]^\theta$. 
  Since $b$ is a cofinal branch through $\mc S$, we can find $x \in A$ such that 
  $z \subseteq x$ and $b \cap x \in S_x$. Then $z \subseteq b \cap x$, and there is 
  $i < \theta$ such that $b \cap x = x \cap D(i, \gamma_x)$; therefore, 
  $z \subseteq D(i, \gamma_x)$, as desired.
\end{proof}

We are now ready to prove Theorem A, asserting that, for a regular cardinal $\kappa \geq \omega_2$, 
$\cNSP_\kappa$ implies $\SSH$ above $\kappa$.

\begin{proof}[Proof of Theorem A]
  By Fact \ref{pp_fact}(1), to establish $\SSH$ above $\kappa$, it suffices to show that $\pp(\mu) 
  = \mu^+$ for every singular cardinal $\mu > \kappa$ of countable cofinality. Fix such a $\mu$.
  Next, by Fact \ref{pp_fact}(2), to establish $\pp(\mu) = \mu^+$, it suffices to prove that 
  $\cf(\prod \vec{\mu}, <^*) = \mu^+$ for every increasing sequence of regular cardinals 
  $\vec{\mu} = \langle \mu_i \mid i < \omega \rangle$ converging to $\mu$. Fix such a sequence 
  $\vec{\mu}$. By the proof of \cite[Lemma 2.4]{sharon_viale} (cf.\ also \cite[Lemma 4.4]{arithmetic_paper}), 
  there is a uniform, transitive $\omega$-covering 
  matrix $\mc D = \langle D(i,\beta) \mid i < \omega, ~ \beta < \mu^+ \rangle$ for $\mu^+$ such that 
  $|D(i,\beta)| < \mu_i$ for all $i < \omega$ and $\beta < \mu^+$. By Theorem \ref{nsp_cp_thm} 
  and the assumption that $\cNSP_\kappa$ holds, we know that 
  $\CP(\mc D)$ holds, and then, by Theorem \ref{cp_ssh_thm}, we have $\cf(\prod \vec{\mu}, <^*) 
  = \mu^+$, as desired.
\end{proof}

\section{General systems} \label{system_sec}

We now move to the more general setting of systems indexed by arbitrary directed partial orders.
Given a partial order $(\Lambda, \leq_\Lambda)$, we will sometimes abuse notation and use the symbol 
$\Lambda$ to denote the partial order. If a partial order is denoted by $\Lambda$, it should be 
understood that its order relation is denoted by $\leq_\Lambda$. The strict portion of $\leq_\Lambda$ 
will be denoted by $<_\Lambda$. Given $u \in \Lambda$, let $u^\uparrow$ denote $\{v \in \Lambda \mid 
u <_\Lambda v\}$.

\begin{definition}
  Suppose that $\Lambda$ is a partial order and $\kappa$ is an infinite cardinal. We say that 
  $\Lambda$ is \emph{$\kappa$-directed} if every element of $\power_\kappa \Lambda$ has an upper 
  bound, i.e., for every $x \in \power_\kappa \Lambda$, there is $v \in \Lambda$ 
  such that $u \leq_{\Lambda} v$ for all $u \in x$. We say that $\Lambda$ is \emph{directed} 
  if it is $\aleph_0$-directed; equivalently, for all $u,v \in \Lambda$, there is $w \in \Lambda$ 
  such that $u,v \leq_\Lambda w$.
\end{definition}

\begin{definition}
  Suppose that $\Lambda$ is a directed partial order. The \emph{directedness} of $\Lambda$, denoted 
  $d_\Lambda$, is the largest cardinal $\kappa$ such that $\Lambda$ is $\kappa$-directed. It is readily 
  verified that this is well-defined and that $d_\Lambda$ is a regular cardinal for every directed partial 
  order $\Lambda$.
\end{definition}

\begin{definition}
	Let $R$ be a binary relation on a set $X$. For $x,y \in X$, we will typically write 
	$x <_R y$ to denote $(x,y) \in R$ and $x \leq_R y$ to denote the statement 
	\[
	  (x,y) \in R \text{ or } x = y.
	\] 
	Two elements $x$ and $y$ of $X$ are said to be
	\emph{$R$-comparable} if either $x \leq_R y$ or $y \leq_R x$. Otherwise, $x$ and $y$ 
	are \emph{$R$-incomparable}.
\end{definition}

\begin{definition} \label{system_def}
  Let $\Lambda$ be a directed partial order. A \emph{$\Lambda$-system} is a structure
  \[
    \mc S = \left \langle \langle S_u \mid u \in \Lambda \rangle, \mc R \right \rangle
  \]
  satisfying the following conditions.
  \begin{enumerate}
    \item $\langle S_u \mid u \in \Lambda \rangle$ is a sequence of pairwise disjoint nonempty sets. 
    We will sometimes refer to $\bigcup_{u \in \Lambda} S_u$ as the \emph{underlying set} of 
    $\mc S$, and we will sometimes simply denote it by $S$. For each $x \in S$, let $\ell(x)$ denote 
    the unique $u \in \Lambda$ such that $x \in S_u$.
    \item $\mc R$ is a nonempty set of binary, transitive relations on $S$.
    \item For all $x,y \in S$ and $R \in \mc R$, if $x <_R y$, then $\ell(x) <_\Lambda \ell(y)$.
    \item \label{treelike_clause} 
    For all $x,y,z \in S$ and $R \in \mc R$, if $x,y <_R z$ and $\ell(x) \leq_{\Lambda} 
    \ell(y)$, then $x \leq_R y$.
    \item \label{completeness_clause}
    For all $(u,v) \in \Lambda^{[2]}$, there are $x \in S_u$, $y \in S_v$, and 
    $R \in \mc R$ such that $x <_R y$.
  \end{enumerate}
  If $\mc S$ is a $\Lambda$-system, then we define $\width(\mc S)$ to be $\max\{\sup\{|S_u| \mid 
  u \in \Lambda\}, |\mc R|\}$. We say that $\mc S$ is a \emph{narrow $\Lambda$-system} if 
  $\width(\mc S)^+ < d_\Lambda$.
  
  If $\mc S = \left \langle \langle S_u \mid u \in \Lambda \rangle, \mc R \right \rangle$ is a 
  $\Lambda$-system, $x,y \in S$, and $R \in \mc R$, then we say that $x$ and $y$ are 
  \emph{$R$-compatible}, denoted $x \parallel_R y$, if there is $z \in S$ such that $x,y \leq_R z$. 
  We say that $x$ and $y$ are \emph{$R$-incompatible}, denoted $x \perp_R y$, if there is no 
  such $z$. Note that, if $x \parallel_R y$ and $\ell(x) \leq_{\Lambda} \ell(y)$, then Clause 
  \ref{treelike_clause} above implies that $x \leq_R y$.
  
  Given $R \in \mc R$, a \emph{branch through $R$ in $\mc S$} is a set 
  $b \subseteq S$ such that, for all $x,y \in b$, we have $x \parallel_R y$ 
  (note that this implies that $|b \cap S_u| \leq 1$ for all $u \in \Lambda$).
  We will sometimes say that $b$ is a \emph{branch in $\mc S$} to mean that there is 
  $R \in \mc R$ such that $b$ is a branch through $R$ in $\mc S$.
  A branch $b$ is said to be \emph{cofinal} if $\{u \in \Lambda \mid b \cap S_u \neq \emptyset\}$ 
  is cofinal in $\Lambda$.
\end{definition}

\begin{remark}
  The concrete $\power_\kappa \lambda$-systems of Section~\ref{concrete_sec} are indeed 
  special cases of Definition~\ref{system_def}:
  suppose that $\mc S = \langle S_x \mid x \in A \rangle$ is a concrete $\power_\kappa 
  \lambda$-system. Then there is a natural way to view $\mc S$ as an $(A, \subsetneq)$-system 
  in the sense of Definition~\ref{system_def}. Namely, for each $x \in A$, let 
  $S'_x := \{x\} \times S_x$, and define a binary relation $R$ on $\bigcup_{x \in A} 
  S'_x$ by letting $(x,t) <_R (y,s)$ iff $x \subsetneq y$ and $y \cap t = x$. Then 
  $\mc S' := \langle \langle S'_x \mid x \in A \rangle, \{R\} \rangle$ is readily verified to be 
  an $(A, \subsetneq)$-system in the sense of Definition~\ref{system_def}, and cofinal branches 
  through $\mc S$ in the sense of Definition~\ref{concrete_system_def} naturally correspond to 
  cofinal branches through $\mc S'$ in the sense of Definition~\ref{system_def}. 
\end{remark}

\begin{definition}
  Let $\Lambda$ be a directed partial order. We say that the \emph{$\Lambda$-narrow system property} 
  (denoted $\NSP(\Lambda)$) holds if every narrow $\Lambda$-system has a cofinal branch.
\end{definition}

Since all of the questions considered in this paper become trivial when addressing systems indexed by 
partial orders with maximal elements, we will always assume when working with arbitrary
$\Lambda$-systems that $\Lambda$ has no maximal element, even when this assumption is not explicitly 
stated.

For notational simplicity, we often prefer to work with systems having only one relation. The 
following proposition shows that, in the context of questions about the existence of narrow 
$\Lambda$-systems without cofinal branches, this involves no loss of generality.

\begin{proposition} \label{single_relation_prop}
  Suppose that $\Lambda$ is a directed partial order and 
  \[
    \mc S = \left \langle \langle S_u \mid u \in \Lambda \rangle, \mc R \right \rangle
  \] 
  is a $\Lambda$-system. 
  Then there is a $\Lambda$-system $\mc S' = \left \langle \langle S'_u \mid u \in \Lambda \rangle, 
  \mc R' \right \rangle$ such that
  \begin{itemize}
    \item $|\mc R'| = 1$;
    \item $\width(\mc S') = \width(\mc S)$;
    \item $\mc S'$ has a cofinal branch if and only if $\mc S$ has a cofinal branch.
  \end{itemize}
\end{proposition}

\begin{proof}
  For each $u \in \Lambda$, let $S'_u := S_u \times \mc R$, and let $\mc R'$ consist of a single 
  binary relation, $<'$. For all $x_0,x_1 \in S$ and $R_0, R_1 \in \mc R$, let $(x_0,R_0) <' (x_1,R_1)$ 
  if and only if $R_0=R_1$ and $x_0 <_{R_0} x_1$ (in $\mc S$). It is readily verified that $\mc S'$ thus defined 
  is a $\Lambda$-system and $\width(\mc S') = \width(\mc S)$. If $R \in \mc R$ and $b \subseteq S$ is a 
  cofinal branch through $R$ in $\mc S$, then $b' := \{(x,R) \mid x \in b\}$ is a cofinal branch 
  in $\mc S'$. Conversely, if $d'$ is a cofinal branch in $\mc S'$, then there must be a single 
  $R \in \mc R$ such that every element of $d'$ is of the form $(x,R)$ for some $x \in S$. Then 
  $d := \{x \in S \mid (x,R) \in d'\}$ is a cofinal branch through $R$ in $\mc S$.
\end{proof}
	
The following basic proposition is reminiscent of K\"{o}nig's Infinity Lemma, asserting that every 
infinite finitely-branching tree has an infinite branch.

\begin{proposition} \label{finite_width_prop}
  Suppose that $\Lambda$ is a directed partial order and $\mc S$ is a $\Lambda$-system with finite width. 
  Then $\mc S$ has a cofinal branch.
\end{proposition}

\begin{proof}
  By Proposition \ref{single_relation_prop}, we can assume that $\mc S$ has a single relation, which 
  we will denote by $R$. Since $\width(\mc S)$ is finite, we can fix an $n < \omega$ such that 
  $|S_u| \leq n$ for all $u \in \Lambda$. Enumerate each $S_u$ as $\langle x_{u, k} \mid k < n \rangle$, 
  with repetitions if necessary (i.e., if $|S_u| < n$). Since $\Lambda$ is 
  directed, $\mc F := \{u^\uparrow \mid u \in \Lambda\}$ is a filter over $\Lambda$. Let $\mc U$ be 
  an ultrafilter over $\Lambda$ extending $\mc F$.
  
  Temporarily fix $u \in \Lambda$. Since $\mc S$ is a $\Lambda$-system, it follows that, for every 
  $v \in u^\uparrow$, we can find (not necessarily unique) $j(u,v), k(u,v) < n$ such that 
  $x_{u,j(u,v)} <_R x_{v,k(u,v)}$. Since $\mc U$ is an ultrafilter extending $\mc F$, we can then find 
  fixed numbers $j(u), k(u) < n$ such that the set
  \[
    X_u := \{v \in u^\uparrow \mid (j(u,v), k(u,v)) = (j(u),k(u))\}
  \]
  is in $\mc U$.\footnote{We are using the implicit assumption that $\Lambda$ has no maximal element 
  to ensure that we can find such $j(u)$ and $k(u)$.} 
  We can then find fixed numbers $j^*,k^* < n$ such that the set
  \[
    Y := \{u \in \Lambda \mid (j(u),k(u))=(j^*,k^*)\}
  \]
  is in $\mc U$. In particular, $Y$ is cofinal in $\Lambda$. Let $b := \{x_{u,j^*} \mid u \in Y\}$. 
  Since $Y$ is cofinal in $\Lambda$, in order to show that $b$ is a cofinal branch in $\mc S$ it suffices 
  to show that, for all $u_0, u_1 \in Y$, we have $x_{u_0,j^*} \parallel_R x_{u_1, j^*}$. To 
  this end, fix such $u_0, u_1$. Since $X_{u_0}, X_{u_1} \in \mc U$, we can 
  fix $v \in X_{u_0} \cap X_{u_1}$. then $x_{u_0, j^*}, x_{u_1, j^*}, <_{R} x_{v,k^*}$, 
  so $x_{u_0,j^*} \parallel_R x_{u_1, j^*}$, as desired.
\end{proof}

An analogous result holds at strongly compact cardinals:

\begin{proposition} \label{strongly_compact_prop}
  Suppose that $\kappa$ is a strongly compact cardinal, $\Lambda$ is a directed partial order with 
  $d_\Lambda \geq \kappa$, and $\mc S$ is a $\Lambda$-system such that $\width(\mc S) < \kappa$. Then 
  $\mc S$ has a cofinal branch.
\end{proposition}

\begin{proof}
  The proof is essentially the same as that of Proposition \ref{finite_width_prop} and is thus mostly 
  left to the reader. We remark only that, due to the fact that $d_\Lambda \geq \kappa$, the filter 
  $\mc F := \{u^\uparrow \mid u \in \Lambda\}$ is $\kappa$-complete and, since $\kappa$ is strongly 
  compact, it can be extended to a $\kappa$-complete ultrafilter $\mc U$ over $\Lambda$. The rest of the 
  proof is precisely as in Proposition \ref{finite_width_prop}.
\end{proof}

As mentioned already, classical narrow systems were introduced by Magidor and Shelah in the context of 
the study of the tree property at successors of singular cardinals; their first application came in 
the proof that, if $\mu$ is a singular limit of strongly compact cardinals, then the tree property holds 
at $\mu^+$ \cite[Theorem 3.1]{magidor_shelah}. To help get a feel for the utility of narrow $\Lambda$-systems, 
we present here the analogous result in the more general setting. We first need to recall the notion of 
a $\kappa$-$\Lambda$-tree for an arbitrary directed partial order $\Lambda$.

\begin{definition}[\cite{kurepa_paper}]
  Let $\Lambda$ be a directed partial order. A \emph{$\Lambda$-tree} is a structure 
  $\mc T = (\langle T_u \mid u \in \Lambda \rangle, <_{\mc T})$ such that the 
  following conditions all hold.
  \bce[(i)]
    \item $\langle T_u \mid u \in \Lambda \rangle$ is a sequence of nonempty, pairwise disjoint sets.
    \item $<_{\mc T}$ is a transitive partial ordering on $\bigcup_{u \in \Lambda} T_u$.
    \item For all $u,v \in \Lambda$, all $s \in T_u$, and all $t \in T_v$, if 
    $s <_{\mc T} t$, then $u <_\Lambda v$.
    \item $<_{\mc T}$ is \emph{tree-like}, i.e., for all $u <_\Lambda v <_\Lambda w$, all 
    $r \in T_u$, all $s \in T_v$ and all $t \in T_w$, if $r, s <_{\mc T} t$, then $r <_{\mc T} s$.
    \item For all $u \leq_\Lambda v$ in $\Lambda$ and all $t \in T_v$, there is a unique $s \in T_u$, 
    denoted $t \restriction u$, such that $s \leq_{\mc T} t$. 
  \ece 
  For a cardinal $\kappa$, we say that $\mc T$ is a $\kappa$-$\Lambda$-tree if, in addition to the 
  above requirements, we have $|T_u| < \kappa$ for all $u \in \Lambda$.
  
  Suppose that $\mc T$ is a $\Lambda$-tree. A \emph{cofinal branch} through $\mc T$ is a function 
  $b \in \prod_{u \in \Lambda} T_u$ such that, for all $u <_\Lambda v$ in $\Lambda$, we have 
  $b(u) <_{\mc T} b(v)$. The \emph{$(\kappa, \Lambda)$-tree property}, denoted $\TP_\kappa(\Lambda)$, 
  is the assertion that every $\kappa$-$\Lambda$-tree has a cofinal branch. We let 
  $\TP(\Lambda)$ denote $\TP_{d_\Lambda}(\Lambda)$.
\end{definition}

\begin{theorem} \label{strongly_compact_tp_thm}
  Suppose that $\mu$ is a singular limit of strongly compact cardinals and $\Lambda$ is a 
  $\mu^+$-directed partial order. Then $\TP_{\mu^+}(\Lambda)$ holds.
\end{theorem}

\begin{proof}
  Let $\theta := \cf(\mu)$, and let $\langle \mu_i \mid i < \theta \rangle$ be an increasing sequence of 
  strongly compact cardinals, converging to $\mu$, with $\mu_0 > \theta$. Let $\mc T = \langle \langle T_u \mid 
  u \in \Lambda \rangle, <_{\mc T} \rangle$ be a $\Lambda$-tree with $|T_u| \leq \mu$ for all $u \in \Lambda$. 
  We first show that $\mc T$ has a narrow subsystem indexed by a cofinal subset of $\Lambda$.
  
  \begin{claim}
    There is a cofinal $\Gamma \subseteq \Lambda$ and, for each $u \in \Gamma$, a nonempty $S_u \subseteq T_u$ 
    such that $\mc S := \langle \langle S_u \mid u \in \Gamma \rangle, \{<_{\mc S}\} \rangle$ is a narrow 
    $\Gamma$-system, where $<_{\mc S}$ is the restriction of $<_{\mc T}$ to $\bigcup_{u \in \Gamma} S_u$.
  \end{claim}  
  
  \begin{proof}
    For all $u \in \Lambda$, enumerate $T_u$ as $\langle t^u_\eta \mid \eta < \mu \rangle$ (with repetitions, 
    if necessary). Fix an elementary embedding $j:V \rightarrow M$ witnessing that $\mu_0$ is 
    $|\Lambda|$-strongly compact. In particular, we have
    \begin{itemize}
      \item $\mathrm{crit}(j) = \mu_0$;
      \item $j(\mu_0) > |\Lambda|$;
      \item there is $W \in M$ such that $W \subseteq j(\Lambda)$, $|W|^M < j(\mu_0)$, and 
      $j``\Lambda \subseteq W$.
    \end{itemize}
    Let $j(\mc T) = \mc T' = \langle T'_v \mid v \in j(\Lambda) \rangle$. Since $|W|^M < j(\mu_0) < j(\mu^+)$ and 
    $j(\Lambda)$ is $j(\mu^+)$-directed in $M$, we can find $z \in j(\Lambda)$ such that $w <_{j(\Lambda)} 
    z$ for all $w \in W$; in particular, $j(u) <_{j(\Lambda)} z$ for all $u \in \Lambda$. Choose an arbitrary 
    $t \in T'_z$. For each $u \in \Lambda$, enumerate $T'_{j(u)}$ as $\langle (t')^{j(u)}_\eta \mid 
    \eta < j(\mu) \rangle$. For each $u \in \Lambda$, there is $i_u < \theta$ and $\eta_u < j(\mu_{i_u})$ 
    such that $(t')^{j(u)}_{\eta_u} <_{j(\mc T)} t$. Since $\Lambda$ is $\mu^+$-directed, we can find a fixed 
    $i < \theta$ and a cofinal $\Gamma \subseteq \Lambda$ such that $i_u = i$ for all $u \in \Gamma$. 
    Then, for all $u <_\Lambda v$, both in $\Gamma$, we have $(t')^{j(u)}_{\eta_u}, (t')^{j(v)}_{\eta_v} 
    <_{j(\mc T)} t$, and hence $ (t')^{j(u)}_{\eta_u} <_{j(\mc T)} (t')^{j(v)}_{\eta_v}$. In particular, 
    \[
      M \models \exists \eta, \xi < j(\mu_i) \left[ (t')^{j(u)}_\eta <_{j(\mc T)} 
      (t')^{j(v)}_\xi \right],
    \]
    as witnessed by $\eta = \eta_u$ and $\xi = \xi_v$. By elementarity, we have
    \[
      V \models \exists \eta, \xi < \mu_i \left[ t^u_\eta <_{\mc T} t^v_\xi \right].
    \]
    
    It is now readily verified that, if we let $S_u := \{t^u_\eta \mid \eta < \mu_i\}$ for all $u \in 
    \Gamma$, then $\mc S$ as in the statement of the claim is indeed a narrow $\Gamma$-system: clauses 
    (1)--(4) of Definition~\ref{system_def} are immediate, and clause (5) follows from the elementarity 
    argument in the previous paragraph.
  \end{proof}
  
  Let $\mc S$ be as given by the claim, and let $i < \theta$ be such that $\width(\mc S) < \mu_i$. Then we can apply 
  Proposition~\ref{strongly_compact_prop} with $\mu_i$ and $\Gamma$ in place of $\kappa$ and $\Lambda$, 
  respectively, to conclude that $\mc S$ has a cofinal branch, $b \subseteq S$. This readily gives rise to 
  a cofinal branch $b' \in \prod_{u \in \Lambda} T_u$ through $\mc T$: for each $u \in \Lambda$, find 
  $v \in \Gamma$ such that $u \leq_\Lambda v$ and $b \cap S_v \neq \emptyset$. 
  Let $s$ be the unique element of $b \cap S_v$, and then
  let $b'(u) := s \restriction u$. 
\end{proof}

\section{Subadditive colorings} \label{subbadditive_sec}

In this brief section, we highlight a connection between the narrow system properties introduced 
in the previous section and the existence of certain strongly unbounded subadditive colorings 
on arbitrary directed partial orders. Such colorings on \emph{ordinals} have been extensively studied 
and have proven to be useful in a variety of contexts (cf.\ \cite{knaster_iii}). Here we generalize the 
notion to arbitrary directed orders, show that instances of the narrow system property imply 
the nonexistence of certain strongly unbounded subadditive colorings, and then show that the 
nonexistence of certain strongly unbounded subadditive colorings can replace the narrow system property 
hypothesis in the statement of Theorem A.

\begin{definition}
  Suppose that $\Lambda$ is a directed partial order and $\theta$ is an infinite regular cardinal. 
  Let $c:\Lambda^{[2]} \rightarrow \theta$ be a function.
  \begin{enumerate}
    \item We say that $c$ is \emph{subadditive} if, for all triples $u <_\Lambda v <_\Lambda w$ from 
    $\Lambda$, we have
    \begin{enumerate}
      \item $c(u,w) \leq \max\{c(u,v), c(v,w)\}$; and
      \item $c(u,v) \leq \max\{c(u,w), c(v,w)\}$.
    \end{enumerate}
    \item We say that $c$ is \emph{strongly unbounded} if, for every cofinal subset $\Gamma \subseteq 
    \Lambda$, $c``\Gamma^{[2]}$ is unbounded in $\theta$.
  \end{enumerate}
\end{definition}

\begin{proposition}
  Suppose that $\Lambda$ is a directed partial order, $\theta$ is an infinite regular cardinal, and 
  $c:\Lambda^{[2]} \rightarrow \theta$ is a strongly unbounded subadditive function. Then there 
  is a $\Lambda$-system with width $\theta$ and no cofinal branch.
\end{proposition}

\begin{proof}
  We will define a $\Lambda$-system $\mc S = \left \langle \langle S_u \mid u \in \Lambda \rangle, 
  \mc R \right \rangle$. First, for each $u \in \Lambda$, let $S_u := \{u\} \times \theta$, and let 
  $\mc R = \{R\}$ consist of a single relation. Now, for $(u,v) \in \Lambda^{[2]}$ and 
  $i,j < \theta$, set $(u,i) <_R (v,j)$ if and only if $i = j$ and $c(u,v) \leq i$. The fact that 
  $\mc S$ is a $\Lambda$-system follows from the subadditivity of $c$, and it is evident that 
  $\width(\mc S) = \theta$. Now suppose for sake of contradiction that $\mc S$ has a cofinal branch, 
  $b$. Then $b$ is necessarily of the form $\{(u,i) \mid u \in \Gamma\}$ for some fixed 
  $i < \theta$ and some cofinal $\Gamma \subseteq \Lambda$. But then $c``\Gamma^{[2]} 
  \subseteq i+1$, contradicting the fact that $c$ is strongly unbounded.
\end{proof}

\begin{corollary} \label{nsp_sub_cor}
  Suppose that $\Lambda$ is a directed partial order and $\NSP(\Lambda)$ holds. Then, for every 
  infinite regular cardinal $\theta$ with $\theta^+ < d_\Lambda$, there does not exist a strongly 
  unbounded subadditive coloring $c:\Lambda^{[2]} \rightarrow \theta$. \qed 
\end{corollary}

We now show that the nonexistence of strongly unbounded subadditive colorings from 
$(\power_\kappa \mu^+)^{[2]}$ to $\cf(\mu)$ can be used in place of $\cNSP(\power_\kappa \mu^+)$
to yield the conclusion of Theorem~\ref{nsp_cp_thm}.

\begin{theorem} \label{sub_cp_thm}
	Suppose that $\kappa < \mu$ are infinite cardinals such that
	\begin{itemize}
		\item $\cf(\mu) < \kappa$; and
		\item there does not exist a strongly unbounded subadditive coloring 
		\[
		  c:(\power_\kappa \mu^+)^{[2]} \rightarrow \cf(\mu).
		\]
	\end{itemize}
	Then $\CP(\mc D)$ holds for every uniform, transitive $\cf(\mu)$-covering matrix 
	for $\mu^+$.
\end{theorem}

\begin{proof}
	Let $\theta := \cf(\mu)$ and $\lambda := \mu^+$, and let 
	$\mc D = \langle D(i, \beta) \mid i < \theta, ~ \beta < \lambda^+ \rangle$ be a uniform, 
	transitive $\theta$-covering matrix for $\lambda$. By Theorem~\ref{downward_coherence_thm}, 
	$\mc D$ has the property that, for every 
	$x \in \mathscr P_\kappa \lambda$, there is $\gamma_x < \lambda$ such that, for all 
	$\beta \in [\gamma_x, \lambda)$, there is $i < \theta$ such that, for all 
	$j \in [i, \theta)$, we have $x \cap D(j, \beta) = x \cap D(j, \gamma_x)$.
	For each $x \in \mathscr P_\kappa \lambda$ and each $j < \theta$, let 
	$x_j := x \cap D(j, \gamma_x)$. Note that $x = \bigcup_{j < \theta} x_j$.
	
	\begin{claim} \label{agreement_claim}
		For all $(x,y) \in (\mathscr P_\kappa \lambda)^{[2]}$, there is $i < \theta$ such that, 
		for all $j \in [i, \theta)$, we have $x_j = y_j \cap x$.
	\end{claim}

	\begin{proof}
		Fix $(x,y) \in (\mathscr P_\kappa \lambda)^{[2]}$, and let $\gamma := \max\{\gamma_x, 
		\gamma_y\}$. Then, by definition of $x_j$ and $y_j$ and the choice of $\gamma_x$ and 
		$\gamma_y$, there is $i < \theta$ such that, for all 
		$j \in [i, \theta)$, we have $x_j = x \cap D(j, \gamma)$ and 
		$y_j = y \cap D(j, \gamma)$. But then, for all $j \in [i, \theta)$, we have 
		$y_j \cap x = D(j, \gamma) \cap x = x_j$, as desired.
	\end{proof}

	Now define a function $c:(\power_\kappa \lambda)^{[2]} \rightarrow \theta$ by letting $c(x,y)$ be the least 
	$i < \theta$ as in Claim \ref{agreement_claim} for all 
	$(x,y) \in (\mathscr P_\kappa \lambda)^{[2]}$.
	
	\begin{claim}
		$c$ is subadditive.
	\end{claim}

	\begin{proof}
		Fix $x \subsetneq y \subsetneq z$ in $\mathscr P_\kappa\lambda$, and fix $j < \theta$. 
		First, if $j \geq \max\{c(x,y), c(y,z)\}$, then $z_j \cap y = y_j$ and $y_j \cap x = x_j$. 
		It follows that $z_j \cap x = x_j$, and from this we can conclude that 
		$c(x,z) \leq \max\{c(x,y), c(y,z)\}$.
		
		Second, if $j \geq \max\{c(x,z), c(y,z)\}$, then $z_j \cap y = y_j$ and $z_j \cap x = x_j$. 
		It then follows that $y_j \cap x = (z_j \cap y) \cap x = z_j \cap x = x_j$, and again we 
		can conclude that $c(x,y) \leq \max\{c(x,z), c(y,z)\}$. Therefore, $c$ is subadditive.
	\end{proof}

	By assumption, $c$ cannot be strongly unbounded. Therefore, there is a $\subseteq$-cofinal 
	$X \subseteq \mathscr P_\kappa \lambda$ and an $i < \theta$ such that $c(x,y) \leq i$ for all 
	$x \subsetneq y$ in $X$.
	
	\begin{claim}
		For all $j \in [i, \theta)$ and all $x, y \in X$, we have $x_j \cap y = y_j \cap x$.
	\end{claim}

	\begin{proof}
		Fix such $j$, $x$, and $y$, and find $z \in X$ such that $x \cup y \subseteq z$. 
		Then $x_j = z_j \cap x$ and $y_j = z_j \cap y$, so $x_j \cap y = (z_j \cap x) \cap y = 
		(z_j \cap y) \cap x = y_j \cap x$. 
	\end{proof}

	For all $j \in [i, \theta)$, let $A_j = \bigcup_{x \in X} x_j$. It follows immediately from 
	the previous claim that, for all $x \in X$, we have $A_j \cap x = x_j$.
	
	\begin{claim}
		There is $j \in [i, \theta)$ such that $A_j$ is unbounded in $\lambda$.
	\end{claim}

	\begin{proof}
		If not, then, for every $j \in [i, \theta)$, there would be $\beta_j < \lambda$ such that 
		$A_j \subseteq \beta_j$. Let $\beta := \sup\{\beta_j \mid j \in [i, \theta)\} < \lambda$, 
		and find $x \in X$ such that $\beta \in x$. Then, for all large enough $j < \theta$, 
		we must have $\beta \in x_j$ and hence $\beta \in A_j$, contradicting the fact that 
		$A_j \subseteq \beta_j \subseteq \beta$.
	\end{proof}

	Fix $j \in [i, \theta)$ such that $A_j$ is unbounded in $\lambda$. We claim that 
	$A_j$ witnesses $\CP(\mc D)$. To this end, fix $w \in [A_j]^\theta$. Let 
	$x \in X$ be such that $w \subseteq x$. Then $w \subseteq A_j \cap x = x_j 
	\subseteq D(j, \gamma_x)$, so $[A_j]^\theta$ is indeed covered by $\mc D$, as desired.
\end{proof}

\begin{corollary} \label{sub_ssh_cor}
  Suppose that $\kappa \geq \omega_2$ is a regular cardinal and, for every singular cardinal 
  $\mu > \kappa$ of countable cofinality, there does not exist a strongly unbounded subadditive 
  coloring $c:(\power_\kappa \mu^+)^{[2]} \rightarrow \omega$. Then $\SSH$ holds above $\kappa$.
\end{corollary}

\begin{proof}
  By Theorem \ref{sub_cp_thm}, the hypothesis implies that, for every singular cardinal $\mu > \kappa$ 
  of countable cofinality and every uniform, transitive $\omega$-covering matrix $\mc D$ for $\mu^+$, 
  we have $\CP(\mc D)$. Then $\SSH$ above $\kappa$ follows exactly as in proof of Theorem A at the 
  end of Section \ref{concrete_sec}.
\end{proof}

\section{A preservation lemma} \label{preservation_sec}

The remainder of the paper is dedicated to the proof of Theorem B, our global consistency result. 
In this section, we prove a technical preservation lemma indicating that if a sufficiently closed 
forcing adds a rich set of branches to a narrow $\Lambda$-system, then that system necessarily 
has a cofinal branch in the ground model. The lemma is a generalization of 
\cite[Lemma 4.3]{narrow_systems} (which itself is a slight improvement on a previous result of 
Sinapova \cite[Theorem 14]{sinapova_tp}) from the context of classical (ordinal-indexed) narrow systems 
to the context of narrow $\Lambda$-systems for arbitrary directed orders $\Lambda$. We first need a 
preliminary definition.

\begin{definition}
  Suppose that $\Lambda$ is a directed partial order and 
  \[
    \mc S = \left \langle \langle S_u \mid u \in \Lambda \rangle, \mc R \right \rangle
  \] 
  is a $\Lambda$-system with $\width(\mc S) = \theta$. Then a \emph{full set of branches in $\mc S$} is a set 
  $\{b_i \mid i < \theta\}$ such that
  \begin{itemize}
    \item for all $i < \theta$, $b_i$ is a branch in $\mc S$;
    \item for all $u \in \Lambda$, there is $i < \theta$ such that $b_i \cap S_u \neq \emptyset$.
  \end{itemize}
\end{definition}

\begin{proposition} \label{full_cofinal_prop}
  Suppose that $\Lambda$ is a directed partial order, 
  \[
    \mc S = \left \langle \langle S_u \mid u \in \Lambda \rangle, \mc R \right \rangle
  \]
  is a $\Lambda$-system with $\width(\mc S) = \theta < d_\Lambda$, and $\{b_i \mid i < \theta\}$ is 
  a full set of branches in $\mc S$. Then there is $i < \theta$ such that $b_i$ is a cofinal branch 
  in $\mc S$.
\end{proposition}

\begin{proof}
  Suppose not. Then, for every $i < \theta$, there is $u_i \in \Lambda$ such that $b_i \cap S_v = 
  \emptyset$ for all $v \in u_i^\uparrow$. Since $d_\Lambda > \theta$, we can find $u^* \in \Lambda$ 
  such that $u_i \leq_{\Lambda} u^*$ for all $i < \theta$. Since $\Lambda$ has no maximal element, 
  $(u^*)^\uparrow \neq \emptyset$. However, for all $v \in (u^*)^\uparrow$ and all $i < \theta$, 
  we have $b_i \cap S_v = \emptyset$, contradicting the fact that $\{b_i \mid i < \theta\}$ is a full 
  set of branches.
\end{proof}

We are now ready for the main preservation lemma.

\begin{lemma} \label{preservation_lemma}
  Suppose that $\Lambda$ is a directed partial order, $\mc S$ is a narrow $\Lambda$-system, 
  $\theta = \width(S)$, $\P$ is a $\theta^+$-closed forcing poset, and 
  \[
    \Vdash_{\P}``\text{there is a full set of branches in } \mc S".
  \]
  Then, in $V$, there is a cofinal branch in $\mc S$.
\end{lemma}

\begin{proof}
  Suppose for sake of contradiction that there is no cofinal branch in $\mc S$. By assumption, 
  we can fix $\P$-names $\{\dot{b}_i \mid i < \theta\}$ such that 
  \[
    \Vdash_{\P}``\{\dot{b}_i \mid i < \theta\} \text{ is a full set of branches in } \mc S".
  \]
  Using the $\theta^+$-closure of $\P$, construct a decreasing sequence $\langle p_i \mid i < \theta 
  \rangle$ of conditions in $\P$ such that, for each $i < \theta$:
  \begin{itemize}
    \item there is $R_i \in \mc R$ such that $p_i \Vdash_{\P} ``\dot{b}_i \text{ is a branch through } 
    R_i"$;
    \item $p_i$ decides the truth value of the statement $``\dot{b}_i \text{ is a cofinal branch in } 
    \mc S"$;
    \item if $p_i \Vdash_{\P} ``\dot{b}_i \text{ is not a cofinal branch}"$, then there is $u_i \in 
    \Lambda$ such that $p_i \Vdash_{\P} ``\forall v \in u^\uparrow ~ (\dot{b}_i \cap S_v = \emptyset)"$.
  \end{itemize}
  Again using the $\theta^+$-closure of $\P$, let $p^*$ be a lower bound for $\langle p_i \mid i < \theta 
  \rangle$. Let 
  \[
    A := \{i < \theta \mid p_i \Vdash_{\P}``\dot{b}_i \text{ is a cofinal branch}"\}
  \]
  and, using the fact that $d_\Lambda > \theta$, find a $u^* \in \Lambda$ such that $u_i \leq_\Lambda 
  u^*$ for all $i \in \theta \setminus A$. Note that, by Proposition \ref{full_cofinal_prop}, it must 
  be the case that $A \neq \emptyset$.
  
  \begin{claim} \label{splitting_claim_1}
    Suppose that $p \leq_{\P} p^*$ and  $i \in A$. Then there are $q_0, q_1 \leq_{\P} p$ and 
    $x_0, x_1 \in S$ such that
    \begin{enumerate}
      \item for $\varepsilon < 2$, $q_\varepsilon \Vdash_{\P} ``x_\varepsilon \in \dot{b}_i"$;
      \item $x_0 \perp_{R_i} x_1$.
    \end{enumerate}
  \end{claim}
  
  \begin{proof}
    Suppose not, and let $p$ and $i$ form a counterexample. Let 
    \[
      b := \{x \in S \mid \exists q \leq_{\P} p ~ (q \Vdash_{\P}``x \in \dot{b}_i)"\}.
    \]
    We claim that $b$ is a cofinal branch through $R_i$ in $\mc S$. Since $i \in A$ and 
    $p \leq_{\P} p_i$, it is immediate that $\{u \in \Lambda \mid b \cap S_u \neq \emptyset\}$ is cofinal 
    in $\Lambda$. Also, for all $x,y \in b$, our assumption that $p$ and $i$ form a counterexample 
    to the claim implies that $x \parallel_{R_i} y$. Thus, $b$ is a cofinal branch in $\mc S$, 
    contradicting our assumption that no such branch exists.
  \end{proof}
  
  \begin{claim} \label{splitting_claim_2}
    Suppose that $p_0, p_1 \leq_{\P} p^*$ and $i \in A$. Then there are $q_0 \leq_{\P} p_0$, 
    $q_1 \leq_{\P} p_1$, and $x_0, x_1 \in S$ such that
    \begin{enumerate}
      \item for $\varepsilon < 2$, $q_\varepsilon \Vdash_{\P} ``x_\varepsilon \in \dot{b}_i"$;
      \item $x_0 \perp_{R_i} x_1$. 
    \end{enumerate}
  \end{claim}
  
  \begin{proof}
    First apply Claim \ref{splitting_claim_1} to obtain $q_{0,0},q_{0,1} \leq_{\P} p_0$ and 
    $x_{0,0}, x_{0,1} \in S$ such that $q_{0, \varepsilon} \Vdash_{\P} ``x_{0, \varepsilon} \in 
    \dot{b}_i"$ for $\varepsilon < 2$ and $x_{0,0} \perp_{R_i} x_{0,1}$. Then find $q_1 \leq_{\P} p_1$ 
    and $x_1 \in S$ such that $\ell(x_{0,0}), \ell(x_{0,1}) \leq_{\Lambda} \ell(x_1)$ and 
    $q_1 \Vdash_{\P}``x_1 \in \dot{b}_i"$. It cannot be the case that $x_1$ is $R_i$-compatible with both 
    $x_{0,0}$ and $x_{0,1}$, as otherwise $x_1$ would witness that $x_{0,0}$ and $x_{0,1}$ are 
    $R_i$-compatible. Therefore, we can fix $\varepsilon < 2$ such that $x_{0,\varepsilon} \perp_{R_i} 
    x_1$. Let $q_0 := q_{0,\varepsilon}$ and $x_0 := x_{0, \varepsilon}$. Then $q_0$, $q_1$, $x_0$, and 
    $x_1$ are as desired.
  \end{proof}
  
  \begin{claim} \label{splitting_claim_3}
    Suppose that $p \leq p^*$. Then there are $q_0, q_1 \leq_{\P} p$ and 
    $\{x^i_\varepsilon \mid i \in A, ~ \varepsilon < 2\} \subseteq S$ such that
    \begin{enumerate}
      \item for every $i \in A$ and $\varepsilon < 2$, we have $q_\varepsilon \Vdash_{\P}``x^i_\varepsilon 
      \in \dot{b}_i"$;
      \item for every $i \in A$, we have $x^i_0 \perp_{R_i} x^i_1$.
    \end{enumerate}
  \end{claim}
  
  \begin{proof}
    We recursively build two decreasing sequences $\langle q_{0, i} \mid i < \theta \rangle$ and 
    $\langle q_{1, i} \mid i < \theta \rangle$ from $\P$, together with elements $\{x^i_\varepsilon 
    \mid i \in A, ~ \varepsilon < 2\}$ as follows.
    
    First, let $q_{0,0} = q_{1,0} = p$. If $j < \theta$ is a limit ordinal, $\varepsilon < 2$, and 
    we have defined $\langle q_{\varepsilon, i} \mid i < j \rangle$, then let $q_{\varepsilon, j}$ 
    be any lower bound for $\langle q_{\varepsilon, i} \mid i < j \rangle$. If $i \in \theta \setminus A$, 
    $\varepsilon < 2$, and $q_{\varepsilon, i}$ has been defined, then simply let 
    $q_{\varepsilon, i+1} = q_{\varepsilon, i}$. Finally, suppose that $i \in A$ and we have defined 
    $\langle q_{0, j} \mid j \leq i \rangle$ and $\langle q_{1,j} \mid j \leq i \rangle$. Then apply 
    Claim \ref{splitting_claim_2} to $q_{0,i}$, $q_{1,i}$, and $i$ to obtain
    $q_{0, i+1} \leq_{\P} q_{0,i}$, $q_{1, i+1} \leq_{\P} q_{1,i}$, and $x^i_0, x^i_1 \in S$ such 
    that
    \begin{itemize}
      \item for $\varepsilon < 2$, $q_{\varepsilon, i+1} \Vdash_{\P}``x^i_\varepsilon \in \dot{b}_i"$;
      \item $x^i_0 \perp_{R_i} x^i_1$.
    \end{itemize}
    At the end of the construction, for each $\varepsilon < 2$, let $q_\varepsilon$ be a lower bound 
    for $\langle q_{\varepsilon, i} \mid i < \theta \rangle$. Then $q_0$, $q_1$, and 
    $\{x^i_\varepsilon \mid i \in A, ~ \varepsilon < 2\}$ are as desired.
  \end{proof}
  
  Now use Claim \ref{splitting_claim_3} and the closure of $\P$ to recursively build a tree of conditions 
  $\{ p_\sigma \mid \sigma \in {^{<\theta}}2\}$ and elements $\{x^{\sigma, i}_\varepsilon \mid 
  \sigma \in {^{<\theta}}2, ~ i \in A, ~ \varepsilon < 2\}$ of $S$ as follows. We will maintain the 
  hypothesis that, for all $\tau, \sigma \in {^{<\theta}}2$, if $\tau$ is an initial segment of 
  $\sigma$, then $p_\sigma \leq_{\P} p_\tau$.
  
  Let $p_\emptyset := p^*$. If $\eta < \theta$ is a limit ordinal, $\sigma \in {^{\eta}}2$, and 
  $p_{\sigma \restriction \xi}$ has been defined for every $\xi < \eta$, then let $p_\sigma$ be any 
  lower bound for $\langle p_{\sigma \restriction \xi} \mid \xi < \eta \rangle$. If $\sigma \in 
  {^{<\theta}}2$ and $p_\sigma$ has been defined, then apply Claim \ref{splitting_claim_3} to find 
  $p_{\sigma ^\frown \langle 0 \rangle}, p_{\sigma ^\frown \langle 1 \rangle} \leq p_\sigma$, and 
  $\{x^{\sigma, i}_\varepsilon \mid i \in A, ~ \varepsilon < 2\} \subseteq S$ such that 
  \begin{enumerate}
    \item for every $i \in A$ and $\varepsilon < 2$, we have $p_{\sigma ^\frown \langle \varepsilon 
    \rangle} \Vdash_{\P}``x^{\sigma, i}_\varepsilon \in \dot{b}_i"$;
    \item for every $i \in A$, we have $x^{\sigma, i}_0 \perp_{R_i} x^{\sigma, i}_1$.
  \end{enumerate}
  For each $f \in {^{\theta}}2$, let $p_f$ be a lower bound for $\langle p_{f \restriction \eta} \mid 
  \eta < \theta \rangle$. Choose $B \subseteq {^{\theta}}2$ with $|B| = \theta^+$, and use the fact 
  that $d_{\Lambda} > \theta^+$ to find $v \in \Lambda$ such that 
  $\ell(x^{f \restriction \eta, i}_\varepsilon) <_{\Lambda} v$ for all $f \in B$, $\eta < \theta$, 
  $i \in A$, and $\varepsilon < 2$. We can also assume that $u^* <_\Lambda v$.
  
  For each $f \in B$, use the fact that $\{\dot{b}_i \mid i < \theta\}$ is forced to be a full set of 
  branches in $\mc S$ to find a $q_f \leq_\P p_f$, an $i_f < \theta$, and an $x_f \in S_v$ such 
  that $q_f \Vdash_{\P}``x_f \in \dot{b}_{i_f}"$. Since $u^* <_\Lambda v$ and each $q_f$ extends 
  $p^*$, it must be the case that $i_f \in A$ for all $f \in B$. Since $|B| = \theta^+ > \width(\mc S)$, 
  we can find distinct $f,g \in B$, $i \in A$, and $x \in S$ such that $i_f = i_g = i$ and 
  $x_f = x_g = x$. Let $\eta^* < \theta$ be 
  the least $\eta$ such that $f(\eta) \neq g(\eta)$, and let $\sigma := f \restriction \eta^* = 
  g \restriction \eta^*$. Without loss of generality, assume that $f(\eta^*) = 0$ and 
  $g(\eta^*) = 1$. Then $q_f \leq_{\P} q_{\sigma ^\frown \langle 0 \rangle}$, and therefore 
  $q_f \Vdash_{\P}``x^{\sigma, i}_0 \in \dot{b}_i"$. Similarly, $q_g \Vdash_{\P}``x^{\sigma, i}_1 
  \in \dot{b}_i"$. Since both $q_f$ and $q_g$ extend $p_i$ and force $x$ to be in $\dot{b}_i$, and since 
  $\ell(x^{\sigma, i}_0), \ell(x^{\sigma, i}_1) <_\Lambda v = \ell(x)$, it must be the case 
  that $x^{\sigma, i}_0 <_{R_i} x$ and $x^{\sigma, i}_1 <_{R_i} x$, contradicting the fact that 
  $x^{\sigma, i}_0 \perp_{R_i} x^{\sigma, i}_1$.
\end{proof}

\section{A global consistency result} \label{consistency_sec}

We are finally ready to prove our consistency result. For organizational reasons, it will be helpful 
to have the following definition.

\begin{definition}
  For every infinite regular cardinal $\kappa$, we say that $\kappa$ has the \emph{strong narrow system 
  property}, denoted $\SNSP_\kappa$, if, for every directed partial order $\Lambda$ with 
  $d_\Lambda \geq \kappa$, every $\Lambda$-system $\mc S$ with $\width(\mc{S})^+ < \kappa$ has a cofinal 
  branch.
\end{definition}

\begin{theorem} \label{single_collapse_theorem}
  Let $\mu < \kappa$ be regular uncountable cardinals, with $\kappa$ supercompact, and let $\P := 
  \mathrm{Coll}(\mu, {<}\kappa)$. Then, in $V^{\P}$, $\SNSP_\kappa$ holds and moreover is indestructible 
  under $\kappa$-directed closed set forcing.
\end{theorem}

\begin{proof}
  Let $G$ be $\P$-generic over $V$. Since trivial forcing is $\kappa$-directed closed, it suffices to 
  prove that, if $\Q$ is a $\kappa$-directed closed set forcing in $V[G]$ and $H$ is $\Q$-generic 
  over $V[G]$, then $\SNSP_\kappa$ holds in $V[G][H]$. 
  
  To this end, fix a $\kappa$-directed closed $\Q \in V[G]$ and a $\Q$-generic filter $H$ over 
  $V[G]$. In $V[G][H]$, let $\Lambda$ be a $\kappa$-directed partial order, and let $\mc S = 
  \left \langle \langle S_u \mid u \in \Lambda \rangle, \mc R \right \rangle$ be a 
  $\Lambda$-system with $\width(\mc S) < \mu$. For concreteness, assume that the underlying sets 
  of both $\Q$ and $\Lambda$ are ordinals. We will show that, in $V[G][H]$, there is a cofinal 
  branch in $\mc S$. By Proposition \ref{single_relation_prop}, we can 
  assume that $\mc S$ has a single relation, which we will denote by $R$.
  
  In $V$, let $\dot{\Q}$ be a $\P$-name for $\Q$, and let $\dot{\Lambda}$ be a $\P \ast \dot{\Q}$-name 
  for $\Lambda$. Fix a cardinal $\delta > \kappa$ such that $|\power(\P \ast \dot{\Q})| < \delta$ and 
  $\Vdash_{\P \ast \dot{\Q}}``|\dot{\Lambda}| < \delta"$, and let $j:V \rightarrow M$ be an 
  elementary embedding witnessing 
  that $\kappa$ is $\delta$-supercompact, i.e., $\mathrm{crit}(j) = \kappa$, $j(\kappa) > \delta$, 
  and ${^{\delta}}M \subseteq M$. We have $j(\P) = \mathrm{Coll}(\mu, {<}j(\kappa))$ so, by 
  \cite[Lemma 3]{magidor_reflecting} (cf.\ also \cite[Fact 6.11]{cfm}), 
  the natural complete embedding $\iota$ of $\P$ into $j(\P)$ can be extended to a complete embedding 
  $\iota'$ of $\P \ast \dot{\Q}$ into $j(\P)$ in such a way that the quotient forcing 
  $j(\P)/\iota'[\P \ast \dot{\Q}]$ is $\mu$-closed. Let $\dot{\R}$ be a $\P \ast \dot{\Q}$-name for 
  this quotient forcing. We then have $j(\P) \cong \P \ast \dot{\Q} \ast \dot{\R}$, and 
  $\dot{\R}$ is forced by $\P \ast \dot{\Q}$ to be $\mu$-closed. 
  
  Let $\R$ be the realization of $\dot{\R}$ in $V[G][H]$, and let $K$ be an $\R$-generic filter over 
  $V[G][H]$. Since, for all $p \in \P$, $j(p) = p$, which is naturally identified with 
  $(p, 1_{\dot{\Q}}, 1_{\dot{\R}})$ in $\P \ast \dot{\Q} \ast \dot{\R}$, we have $j``G \subseteq 
  G \ast H \ast K$, so, in $V[G][H][K]$, we can extend $j$ to 
  $j:V[G] \rightarrow M[G][H][K]$.
  
  By the closure of $M$, we know that $j``H \in M[G][H][K]$. Moreover,
  in that model, $j``H$ is a directed subset of $j(\Q)$ with $|j``H| < \delta < j(\kappa)$, and 
  $j(\Q)$ is $j(\kappa)$-directed closed. We can therefore find $q^* \in j(\Q)$ such that 
  $q^* \leq_{j(\Q)} j(q)$ for all $q \in H$. Let $H^+$ be $j(\Q)$-generic over $V[G][H][K]$ with 
  $q^* \in H^+$. Then $j``H \subseteq H^+$, so, in $V[G][H][K][H^+]$, we can extend $j$ one last time 
  to $j:V[G][H] \rightarrow M[G][H][K][H^+]$. 
  
  Again by the closure of $M$, we have $j``\Lambda \in M[G][H][K][H^+]$. Moreover, in that model, 
  $j``\Lambda$ is a subset of $j(\Lambda)$ with $|j``\Lambda| < \delta < j(\kappa)$, and 
  $j(\Lambda)$ is $j(\kappa)$-directed. We can therefore find $v^* \in j(\Lambda)$ such that 
  $j(u) <_{j(\Lambda)} v^*$ for all $u \in \Lambda$. Let $\theta := \width(\mc S)$. Since 
  $\theta < \mu$, we have $\theta = j(\theta) = \width(j(\mc S))$. Write $j(\mc S)$ as 
  $\left \langle \langle S'_v \mid v \in j(\Lambda) \rangle, \{j(R)\} \right \rangle$.
  Enumerate $S'_{v^*}$ as $\langle y_i \mid i < \theta \rangle$, with repetitions if $|S'_{v^*}| < \theta$. 
  For each $i < \theta$, let $b_i := \{x \in S \mid j(x) <_{j(R)} y_i\}$.
  
  We claim that $\{b_i \mid i < \theta\}$ is a full set of branches in $\mc S$. Let us first 
  verify that each $b_i$ is a branch in $\mc S$. To this end, fix $i < \theta$ and 
  $x,y \in b_i$. Then, in $j(\mc S)$, we have 
  $j(x), j(y) <_{j(R)} y_i$, and hence $j(x) \parallel_{j(R)} j(y)$. By elementarity, 
  we have $x \parallel_{R} y$, as desired.
  
  We next verify that, for all $u \in \Lambda$, there is $i < \theta$ such that $b_i \cap S_u \neq 
  \emptyset$. To this end, fix $u \in \Lambda$. Since $j(u) <_{j(\Lambda)} v^*$, clause 
  \ref{completeness_clause} of Definition \ref{system_def} implies that there are $i < \theta$ and 
  $w \in S'_{j(u)}$ such that $w <_{j(R)} y_i$. Since $|S_u| \leq \theta < \kappa$, we have 
  $S'_{j(u)} = j``S_u$, so we can find $x \in S_u$ such that $j(x) = w$. Then $x \in b_i \cap S_u$.
  
  We have thus shown that $\{b_i \mid i < \theta\} \in V[G][H][K][H^+]$ is a full set of branches 
  in $\mc S$. Therefore, since $\theta^+ \leq \mu$, we can apply Lemma \ref{preservation_lemma} in 
  $V[G][H]$ to $\mc S$ and the $\mu$-closed poset $\R \ast j(\Q)$ to conclude that, in $V[G][H]$,
  there is a cofinal branch in $\mc S$, thus completing the proof of the theorem.
\end{proof}

\begin{remark} \label{cnsp_tp_remark}
  Theorem \ref{single_collapse_theorem} provides a way of verifying our earlier claim 
  from Remark \ref{remark_23} that, in general 
  $\cNSP_\kappa$ does not imply $\TP_\kappa$. For example, if 
  $\kappa$ is supercompact and $\P = \Coll(\omega_1, {<}\kappa)$, then, in $V^{\P}$, we have 
  $\kappa = \aleph_2$, and Theorem \ref{single_collapse_theorem} implies that $\cNSP_\kappa$ 
  holds. On the other hand, $\CH$ holds in $V^{\P}$, so there exists 
  an $\aleph_2$-Aronszajn tree, and hence even $\TP(\kappa, \kappa)$ fails.
\end{remark}

Note that the assertion ``$\SNSP_\kappa$ holds for every infinite regular cardinal $\kappa$" is equivalent 
to the assertion ``$\NSP(\Lambda)$ holds for every directed partial order $\Lambda$". The 
following therefore yields Theorem B.

\begin{theorem} \label{global_theorem}
  Suppose that there is a proper class of supercompact cardinals. Then there is a class forcing extension 
  in which $\SNSP_\kappa$ holds for every infinite regular cardinal $\kappa$.
\end{theorem}

\begin{proof}
  Let $\langle \kappa_\eta \mid \eta \in \mathrm{On} \rangle$ be an increasing, continuous sequence 
  of cardinals such that
  \begin{itemize}
    \item $\kappa_0 = \aleph_0$;
    \item if $\eta$ is a limit ordinal (including $0$), then $\kappa_{\eta+1} = \kappa_\eta^+$;
    \item if $\eta$ is a successor ordinal, then $\kappa_{\eta + 1}$ is supercompact.
  \end{itemize}
  We may assume that $\kappa_\eta$ is singular for every nonzero limit ordinal $\eta$; if not, then 
  simply truncate the universe below $\kappa_\eta$ for the least nonzero limit ordinal $\eta$ such 
  that $\kappa_\eta$ is regular (and hence strongly inaccessible).
  
  We now force with a class-length iteration of L\'{e}vy collapses to turn each $\kappa_\eta$ into 
  $\aleph_\eta$. More formally, recursively define posets $\langle \P_\eta \mid \eta \in \mathrm{On} \rangle$ 
  as follows:
  \begin{itemize}
    \item $\P_0$ and $\P_1$ are trivial forcing;
    \item if $\eta$ is a successor ordinal, then $\P_{\eta+1} = \P_\eta \ast \dot{\mathrm{Coll}}
    (\kappa_\eta, {<}\kappa_{\eta+1})$;
    \item if $\eta$ is a nonzero limit ordinal, then $\P_\eta$ is the inverse (i.e., full-support) limit 
    of $\langle \P_\xi \mid \xi < \eta \rangle$.
  \end{itemize}
  For ordinals $\xi < \eta$, let $\dot{\P}_{\xi\eta}$ be a $\P_\xi$-name for the quotient 
  $\P_\eta/\P_\xi$. Then $\dot{\P}_{\xi\eta}$ is a name for a full-support iteration of L\'{e}vy 
  collapses, each of which is forced to be $\kappa_\xi$-directed closed. It follows that 
  $\dot{\P}_{\xi\eta}$ is forced to be $\kappa_\xi$-closed. In particular, 
  $(H(\kappa_\xi))^{V^{\P_\xi}} = (H(\kappa_\xi))^{V^{\P_\eta}}$, so $V^{\P} := 
  \bigcup_{\eta \in \mathrm{On}} V^{\P_\eta}$ is a model of $\ZFC$. Also, standard arguments show that, 
  in $V^{\P}$, we have $\kappa_\eta = \aleph_\eta$ for all $\eta \in \mathrm{On}$.
  
  We claim that $\SNSP_\kappa$ holds in $V^{\P}$ for every infinite regular cardinal 
  $\kappa$. Note that the infinite regular cardinals in $V^{\P}$ are precisely the cardinals 
  $\kappa_\eta$ for which $\eta$ is either $0$ or a successor ordinal. If $\eta \leq 1$, then 
  every system $\mc S$ such that $(\width(\mc{S}))^+ < \kappa_\eta$ has finite width. It therefore 
  follows from Proposition \ref{finite_width_prop} that $\SNSP_{\aleph_0}$ and $\SNSP_{\aleph_1}$ are 
  true in $\ZFC$. 
  
  We next note that, if $\eta$ is a nonzero limit ordinal and $\mc S$ is a system such that 
  $(\width(\mc S))^+ < \kappa_{\eta + 1}$, then there must be a $\xi < \eta$ such that 
  $(\width(\mc S))^+ < \kappa_\xi$. Therefore, $\SNSP_{\kappa_{\eta+1}}$ will follow from 
  the conjunction of $\SNSP_{\kappa_{\xi}}$ for all $\xi < \eta$.
  
  We are therefore left with the task of verifying $\SNSP_{\kappa_{\eta+1}}$ for all successor ordinals 
  $\eta$. To this end, fix a successor ordinal $\eta$ and fix in $V^{\P}$ a directed partial order 
  $\Lambda$ with $d_{\Lambda} \geq \kappa_{\eta+1}$ and a $\Lambda$-system $\mc S$ with 
  $\width(\mc S)^+ < \kappa_{\eta+1}$.
  
  In $V^{P_\eta}$, we have $\kappa_\eta = \aleph_\eta$ and, since $|\P_\eta| < \kappa_{\eta+1}$, 
  we know that $\kappa_{\eta+1}$ is still supercompact. Moreover, $\P_{\eta,\eta+1} = 
  \mathrm{Coll}(\kappa_\eta, {<}\kappa_{\eta+1})$ so, by Theorem \ref{single_collapse_theorem}, 
  $\SNSP_{\kappa_{\eta+1}}$ holds in $V^{\P_{\eta+1}}$ and every $\kappa_{\eta+1}$-directed closed 
  set forcing extension thereof. Let $\zeta > \eta + 1$ be large enough so that 
  $\Lambda$ and $\mc S$ are in $V^{\P_\zeta}$. In $V^{\P_{\eta+1}}$, $\P_{\eta+1, \zeta}$ is 
  $\kappa_{\eta+1}$-directed closed, so $\SNSP_{\kappa_{\eta+1}}$ holds in 
  $V^{\P_\zeta}$. In particular, $\mc S$ has a cofinal branch in $V^{\P_\zeta}$ and hence also in 
  $V^{\P}$.
\end{proof}

We close the paper with what we feel are the most prominent remaining open questions.

\begin{question} \label{open_q}
  Suppose that $\kappa \geq \omega_2$ is a regular cardinal. Does $\TP_\kappa$ imply 
  $\cNSP_\kappa$? More specifically, if $\lambda \geq \kappa$ is a cardinal, does 
  $\TP(\kappa, \lambda)$ imply $\cNSP(\power_\kappa \lambda)$? More generally, suppose that 
  $\Lambda$ is a directed partial order with $d_\Lambda \geq \aleph_2$. Does $\TP(\Lambda)$
  imply $\NSP(\Lambda)$?
\end{question}

By Theorem A, a positive answer to the first part of Question \ref{open_q} would entail a positive 
answer to Question \ref{sch_q} and therefore a negative answer to Question \ref{global_q}. On 
the other hand, a negative answer to Question \ref{open_q} would seem to require genuinely new ideas, 
since the known methods to verify $\TP(\kappa, \lambda)$ in practice inevitably yield 
$\cNSP(\power_\kappa \lambda)$, as well.

\bibliographystyle{plain}
\bibliography{bib}

\begin{thebibliography}{10}

\bibitem{cfm}
James Cummings, Matthew Foreman, and Menachem Magidor.
\newblock Squares, scales and stationary reflection.
\newblock {\em J. Math. Log.}, 1(1):35--98, 2001.

\bibitem{ineffable_tree_property}
James Cummings, Yair Hayut, Menachem Magidor, Itay Neeman, Dima Sinapova, and
  Spencer Unger.
\newblock The ineffable tree property and failure of the singular cardinals
  hypothesis.
\newblock {\em Trans. Amer. Math. Soc.}, 373(8):5937--5955, 2020.

\bibitem{hachtman_sinapova_itp}
Sherwood Hachtman and Dima Sinapova.
\newblock I{TP}, {ISP}, and {SCH}.
\newblock {\em J. Symb. Log.}, 84(2):713--725, 2019.

\bibitem{jech_combinatorial_problems}
Thomas Jech.
\newblock Some combinatorial problems concerning uncountable cardinals.
\newblock {\em Ann. Math. Logic}, 5:165--198, 1972/73.

\bibitem{krueger_sch}
John Krueger.
\newblock Guessing models imply the singular cardinal hypothesis.
\newblock {\em Proc. Amer. Math. Soc.}, 147(12):5427--5434, 2019.

\bibitem{narrow_systems}
Chris Lambie-Hanson.
\newblock Squares and narrow systems.
\newblock {\em J. Symb. Log.}, 82(3):834--859, 2017.

\bibitem{knaster_iii}
Chris Lambie-Hanson and Assaf Rinot.
\newblock Knaster and friends {III}: {S}ubadditive colorings.
\newblock {\em J. Symbolic Logic}, 2023.
\newblock To appear.

\bibitem{kurepa_paper}
Chris Lambie-Hanson and \v{S}\'{a}rka Stejskalov\'{a}.
\newblock Strong tree properties, {K}urepa trees, and guessing models.
\newblock 2022.
\newblock Preprint.

\bibitem{arithmetic_paper}
Chris Lambie-Hanson and \v{S}\'{a}rka Stejskalov\'{a}.
\newblock Guessing models, trees, and cardinal arithmetic.
\newblock 2023.
\newblock Preprint.

\bibitem{magidor_combinatorial_characterization}
Menachem Magidor.
\newblock Combinatorial characterization of supercompact cardinals.
\newblock {\em Proc. Amer. Math. Soc.}, 42:279--285, 1974.

\bibitem{magidor_reflecting}
Menachem Magidor.
\newblock Reflecting stationary sets.
\newblock {\em J. Symbolic Logic}, 47(4):755--771 (1983), 1982.

\bibitem{magidor_shelah}
Menachem Magidor and Saharon Shelah.
\newblock The tree property at successors of singular cardinals.
\newblock {\em Arch. Math. Logic}, 35(5-6):385--404, 1996.

\bibitem{matet_meeting_numbers}
Pierre Matet.
\newblock Meeting numbers and pseudopowers.
\newblock {\em MLQ Math. Log. Q.}, 67(1):59--76, 2021.

\bibitem{mitchell}
William Mitchell.
\newblock Aronszajn trees and the independence of the transfer property.
\newblock {\em Ann. Math. Logic}, 5:21--46, 1972/73.

\bibitem{sharon_viale}
Assaf Sharon and Matteo Viale.
\newblock Some consequences of reflection on the approachability ideal.
\newblock {\em Trans. Amer. Math. Soc.}, 362(8):4201--4212, 2010.

\bibitem{cardinal_arithmetic}
Saharon Shelah.
\newblock {\em Cardinal arithmetic}, volume~29 of {\em Oxford Logic Guides}.
\newblock The Clarendon Press, Oxford University Press, New York, 1994.
\newblock Oxford Science Publications.

\bibitem{sinapova_tp}
Dima Sinapova.
\newblock The tree property at {$\aleph_{\omega+1}$}.
\newblock {\em J. Symbolic Logic}, 77(1):279--290, 2012.

\bibitem{solovay_sch}
Robert~M. Solovay.
\newblock Strongly compact cardinals and the {GCH}.
\newblock In {\em Proceedings of the {T}arski {S}ymposium ({P}roc. {S}ympos.
  {P}ure {M}ath., {V}ol. {XXV}, {U}niv. {C}alifornia, {B}erkeley, {C}alif.,
  1971)}, pages 365--372. Amer. Math. Soc., Providence, R.I., 1974.

\bibitem{specker}
E.~Specker.
\newblock Sur un probl\`eme de {S}ikorski.
\newblock {\em Colloq. Math.}, 2:9--12, 1949.

\bibitem{viale_pfa_sch}
Matteo Viale.
\newblock The proper forcing axiom and the singular cardinal hypothesis.
\newblock {\em J. Symbolic Logic}, 71(2):473--479, 2006.

\bibitem{viale_covering}
Matteo Viale.
\newblock A family of covering properties.
\newblock {\em Math. Res. Lett.}, 15(2):221--238, 2008.

\bibitem{viale_guessing_models}
Matteo Viale.
\newblock Guessing models and generalized {L}aver diamond.
\newblock {\em Ann. Pure Appl. Logic}, 163(11):1660--1678, 2012.

\bibitem{weiss}
Christoph Wei\ss.
\newblock The combinatorial essence of supercompactness.
\newblock {\em Ann. Pure Appl. Logic}, 163(11):1710--1717, 2012.

\end{thebibliography}

\end{document}